\documentclass[11pt,a4paper]{article}

\usepackage[utf8]{inputenc}
\usepackage[T1]{fontenc}
\usepackage{lmodern}
\usepackage{amsmath, amssymb, amsthm}
\usepackage{geometry}
\usepackage{hyperref}
\usepackage{graphicx}
\usepackage{booktabs}
\usepackage{tikz}
\usepackage{pgfplots}
\pgfplotsset{compat=1.17}
\usepackage{float}

\theoremstyle{plain}
\newtheorem{theorem}{Theorem}[section]      
\newtheorem{lemma}[theorem]{Lemma}
\newtheorem{proposition}[theorem]{Proposition}
\newtheorem{corollary}[theorem]{Corollary}

\theoremstyle{definition}
\newtheorem{definition}[theorem]{Definition}

\theoremstyle{remark}
\newtheorem{remark}[theorem]{Remark}

\DeclareMathOperator{\End}{End}

\usepackage[english]{babel} 
\usepackage{microtype}             

\usepackage{mathtools}             
\numberwithin{equation}{section}   
\allowdisplaybreaks                

\hypersetup{
	colorlinks=true,
	linkcolor=blue!50!black,
	citecolor=blue!50!black,
	urlcolor=blue!60!black,
	pdfauthor={Mohammed El Baraka},
	pdftitle={On the Density of CM Points in ell-Isogeny Volcanoes}
}

\newcommand{\F}{\mathbb{F}}
\newcommand{\Z}{\mathbb{Z}}
\newcommand{\Q}{\mathbb{Q}}

\DeclareMathOperator{\Cl}{Cl}

\DeclareMathOperator{\Gal}{Gal}
\DeclareMathOperator{\Li}{Li}
\DeclareMathOperator{\tr}{tr}
\DeclareMathOperator{\disc}{disc}

\DeclareMathOperator{\Ell}{Ell}
\DeclareMathOperator{\dens}{dens}

\title{\textbf{Weighted Distributions of Complex Multiplication Orders in Ordinary Isogeny Classes}}

\author{
	Mohammed El Baraka \& Siham Ezzouak\\
	\small Mathematics Department, Faculty of Sciences Dhar El Mahraz\\
	\small Sidi Mohamed Ben Abdellah University, Fez, Morocco\\
	\small \texttt{mohammed.elbaraka5@usmba.ac.ma}
}

\date{\today}

\providecommand{\keywords}[1]{\par\vspace{.8em}\noindent{\small\textbf{Keywords:} #1}}

\begin{document}
	\maketitle
	
\begin{abstract}
	We develop a global arithmetic framework for the study of endomorphism rings
	inside ordinary elliptic isogeny classes over finite fields. Let $p$ be a prime
	and let $\mathcal I(t,p)$ be an ordinary isogeny class over $\F_p$ with Frobenius
	trace $t$, so that the discriminant $\Delta=t^2-4p$ factors as $\Delta=v^2D_K$.
	In this setting, the possible endomorphism rings are precisely the quadratic
	orders $\mathcal O_f=\Z+f\mathcal O_K$ with $f\mid v$.
	
	Building on Deuring’s correspondence, we express the distribution of these
	orders in terms of weighted class numbers $h^\ast(D)=h(D)/w(D)$, thereby
	obtaining explicit formulas for global distributions across the entire isogeny
	class. This approach overcomes the limitations of the classical local viewpoint,
	in which the endomorphism ring is constant along each level of an
	$\ell$-isogeny volcano. In particular, we introduce weighted exact and cumulative
	distributions of endomorphism rings, which induce canonical laws for the
	$\ell$-adic valuation of conductors and recover the vertical stratification of
	$\ell$-volcanoes in an averaged sense.
	
	On the global side, varying the prime $p$, we relate the existence of curves with
	a prescribed CM order $\mathcal O_D$ to splitting conditions in the associated
	ring class field $L_D$. Using the Chebotarev density theorem, we obtain the
	natural density $1/(2h(D))$ for primes admitting CM by $\mathcal O_D$, providing
	a horizontal distribution law complementary to the vertical conductor
	distribution.
	
	These results establish a unified perspective linking Deuring theory, isogeny
	graph geometry, and class field theory, and provide a natural framework for
	quantitative and algorithmic studies of isogeny classes.
\end{abstract}
	
	

	\section{Introduction}\label{sec:intro}
	
	Let $q=p^n$ be a prime power and let $\ell\neq p$ be a prime.
	The $\ell$-isogeny graph $G_\ell(\F_q)$ is the graph whose vertices are the
	$\F_q$-isomorphism classes of elliptic curves over $\F_q$, and where two vertices
	are joined by an edge if there exists an $\ell$-isogeny between them defined over
	$\F_q$. For ordinary elliptic curves, the connected components of $G_\ell(\F_q)$
	exhibit a highly structured geometry known as \emph{$\ell$-isogeny volcanoes},
	which play a central role in computational number theory and in isogeny-based
	cryptography \cite{Kohel1996,FouquetMorain2002,Sutherland2013Volcanoes}.
	
	\medskip
	\noindent\textbf{Arithmetic structure and limitations of the local viewpoint.}
	It is well known that if $E/\F_q$ is ordinary, then its endomorphism ring
	$\End(E)$ is an order in an imaginary quadratic field $K$. More precisely,
	$\End(E)$ lies between $\Z[\pi_E]$ and the maximal order $\mathcal O_K$, and is
	therefore of the form $\mathcal O_f=\Z+f\mathcal O_K$ for a uniquely determined
	conductor $f\ge 1$ (see \cite{Deuring1941,Waterhouse1969,Silverman2009}).
	
	In the $\ell$-isogeny graph, the $\ell$-adic valuation of this conductor controls
	the vertical position of a vertex inside its volcano component: descending edges
	correspond to multiplication of the conductor by $\ell$, while ascending edges
	correspond to division by $\ell$. As a consequence, each volcano component is
	stratified into levels on which the endomorphism ring is constant.
	
	However, this \emph{local, componentwise} viewpoint has an intrinsic limitation:
	within a fixed volcano component, the endomorphism ring is constant on each level,
	and any notion of “exact density” is therefore trivial, being supported on a single
	level. This makes it impossible to extract meaningful statistical information
	purely from the geometry of individual components.
	
	\medskip
	\noindent\textbf{Global distributions and weighted counting.}
	The main idea of this paper is to move beyond this limitation by introducing a
	\emph{global arithmetic viewpoint}, where distributions are defined across the
	entire ordinary isogeny class rather than within a single volcano component.
	
	Fix a prime $p$ and an ordinary isogeny class $\mathcal I(t,p)$ over $\F_p$.
	Writing the Frobenius discriminant as
	\[
	\Delta = t^2 - 4p = v^2 D_K,
	\]
	the possible endomorphism rings in $\mathcal I(t,p)$ are precisely the quadratic
	orders $\mathcal O_f$ with $f\mid v$.
	
	A key feature of our approach is that we work with \emph{weighted counts},
	expressed in terms of the quantities
	\[
	h^\ast(D) = \frac{h(D)}{w(D)},
	\]
	which naturally arise in Deuring’s correspondence. This avoids the well-known
	pathologies caused by automorphisms (notably for $j=0$ and $1728$), and yields
	arithmetically canonical distributions.
	
	Within this framework, we obtain explicit formulas for the global distribution of
	endomorphism rings inside $\mathcal I(t,p)$, leading to two fundamental invariants:
	\begin{itemize}
		\item a \emph{weighted exact distribution}, describing the contribution of each order
		$\mathcal O_f$;
		\item a \emph{weighted cumulative distribution}, encoding the conductor filtration.
	\end{itemize}
	
	These distributions induce, for each prime $\ell\neq p$, a canonical global law for
	the $\ell$-adic valuation of the conductor, which recovers the vertical structure
	of $\ell$-isogeny volcanoes when averaged over all components.
	
	\medskip
	\noindent\textbf{Horizontal variation and class field theory.}
	On the global side, when the prime $p$ varies, the existence of curves with
	endomorphism ring $\mathcal O_D$ is governed by class field theory.
	More precisely, for $p\nmid D$, the condition that there exists an elliptic curve
	$E/\F_p$ with $\End(E)\simeq \mathcal O_D$ is equivalent to the complete splitting
	of $p$ in the ring class field $L_D$ of $\mathcal O_D$, or equivalently to the
	complete splitting of the Hilbert class polynomial $H_D$ modulo $p$
	\cite{Cox2013,Silverman2009}.
	
	By the Chebotarev density theorem, this yields a natural density
	\[
	\frac{1}{[L_D:\Q]} = \frac{1}{2h(D)}
	\]
	for such primes. This provides a \emph{horizontal distribution law} that complements
	the vertical conductor distribution inside a fixed isogeny class.
	
	\medskip
	\noindent\textbf{Main contributions.}
	The contributions of this paper can be summarised as follows:
	\begin{itemize}
		\item we introduce a global framework for studying endomorphism rings inside
		ordinary isogeny classes via weighted distributions;
		\item we derive explicit formulas for these distributions in terms of weighted
		class numbers, providing a non-trivial refinement of Deuring’s correspondence;
		\item we establish a canonical connection between these global distributions and
		the vertical structure of $\ell$-isogeny volcanoes;
		\item we relate the occurrence of prescribed CM orders across varying primes to
		Chebotarev densities in ring class fields.
	\end{itemize}
	
	\medskip
	\noindent\textbf{Perspective.}
	These results highlight a fundamental duality between:
	\begin{itemize}
		\item vertical, conductor-driven distributions inside a fixed isogeny class;
		\item horizontal, Chebotarev-governed distributions as the base field varies.
	\end{itemize}
	This unified viewpoint provides a natural arithmetic framework for understanding
	the structure of isogeny graphs and suggests further connections with both
	computational and cryptographic applications.
	
	\medskip
	\noindent\textbf{Organisation of the paper.}
	Section~\ref{sec:prelim} recalls the structure of endomorphism rings and
	$\ell$-isogeny volcanoes.
	Section~\ref{sec:isogenyclass-distribution} introduces global weighted distributions
	inside a fixed isogeny class.
	Section~\ref{sec:chebotarev} studies the variation in the prime via ring class fields
	and Chebotarev densities.
\section{Isogeny classes, endomorphism rings, and $\ell$-volcanoes}
\label{sec:prelim}

\subsection{Ordinary isogeny classes and quadratic orders}

Let $q=p^n$ with $p$ prime, and let $E/\F_q$ be an elliptic curve with Frobenius
endomorphism $\pi=\pi_E$. Its characteristic polynomial is
\[
X^2 - tX + q,\qquad t=\tr(\pi),\qquad \#E(\F_q)=q+1-t,
\]
and its discriminant is
\[
\Delta_\pi := t^2-4q < 0.
\]
If $E$ is ordinary, then $K:=\Q(\pi)$ is an imaginary quadratic field and one has
\[
\Z[\pi]\subseteq \End(E)\subseteq \mathcal O_K,
\]
where $\mathcal O_K$ is the maximal order of $K$
(see \cite{Waterhouse1969,Silverman2009}).

Every order $\mathcal O\subseteq \mathcal O_K$ can be written uniquely as
\[
\mathcal O_f = \Z + f\,\mathcal O_K,\qquad f=[\mathcal O_K:\mathcal O_f]\ge 1,
\]
and its discriminant satisfies
\[
\disc(\mathcal O_f)=f^2D_K,
\]
where $D_K$ is the fundamental discriminant of $K$
(see \cite[Ch.~7]{Cox2013}).

\begin{lemma}[Inclusion of quadratic orders]\label{lem:order-inclusion}
	For $f,f'\ge 1$, one has
	\[
	\mathcal O_f\subseteq \mathcal O_{f'} \quad\Longleftrightarrow\quad f'\mid f.
	\]
\end{lemma}

\begin{proof}
	If $f'\mid f$, write $f=f'm$ and obtain
	$\mathcal O_f=\Z+f\mathcal O_K\subseteq \Z+f'\mathcal O_K=\mathcal O_{f'}$.
	Conversely, if $\mathcal O_f\subseteq \mathcal O_{f'}$, then
	$[\mathcal O_K:\mathcal O_{f'}]\mid[\mathcal O_K:\mathcal O_f]$, hence $f'\mid f$.
\end{proof}

\begin{remark}[Terminology]
	Over finite fields, every ordinary elliptic curve has complex multiplication in the
	broad sense that $\End(E)$ is an order in an imaginary quadratic field.
	Throughout this paper, the statement “$E$ has CM by $\mathcal O$” means the precise condition
	$\End(E)\simeq \mathcal O$.
\end{remark}

\subsection{Existence of curves with prescribed endomorphism ring}

Fix an imaginary quadratic order $\mathcal O_f$ of discriminant $D=f^2D_K<0$.
Let
\[
\Ell_{\mathcal O_f}(\F_q)
:=\{\, [E] : E/\F_q \text{ ordinary with } \End(E)\simeq \mathcal O_f \,\}
\]
denote the set of $\F_q$-isomorphism classes with prescribed endomorphism ring.

The following result is a consequence of Deuring's theory and Waterhouse's classification
\cite{Deuring1941,Waterhouse1969}.

\begin{proposition}[Deuring--Waterhouse criterion]\label{prop:DW}
	Let $q=p^n$ with $p\nmid D$. Then $\Ell_{\mathcal O_f}(\F_q)$ is non-empty if and only if
	there exist integers $t,v$ such that
	\[
	4q = t^2 - v^2 D,
	\qquad t \not\equiv 0 \pmod p.
	\]
\end{proposition}

\begin{remark}
	When non-empty, the size of $\Ell_{\mathcal O_f}(\F_q)$ depends on the arithmetic of $q$
	and the factorisation of the Hilbert class polynomial $H_D$ modulo $p$.
	In particular, over $\F_p$ and under complete splitting of $H_D$, one obtains $h(D)$
	distinct $j$-invariants.
\end{remark}

\subsection{$\ell$-isogenies and the volcano stratification}

Fix a prime $\ell\neq p$. Consider the $\ell$-isogeny graph restricted to an ordinary
isogeny class over $\F_q$: vertices are $\F_q$-isomorphism classes of ordinary elliptic curves,
and edges correspond to $\F_q$-rational cyclic $\ell$-isogenies.

The structure of connected components is governed by the $\ell$-adic valuation of the
conductor and the splitting behaviour of $\ell$ in $K$.

Let $D=\disc(\End(E))$ and denote by $\big(\frac{D}{\ell}\big)$ the Kronecker symbol.
The following structure theorem is due to Kohel and subsequent refinements
\cite{Kohel1996,FouquetMorain2002,Sutherland2013Volcanoes}.

\begin{theorem}[Volcano structure]\label{thm:volcano-structure}
	Assume the connected component contains no curves with $j=0$ or $1728$.
	Then it has the structure of an $\ell$-volcano: its vertices are partitioned into levels
	$V_0,\dots,V_d$ such that
	\begin{enumerate}
		\item all vertices in $V_i$ have the same endomorphism ring $\mathcal O_i$, and
		\[
		v_\ell([\mathcal O_K:\mathcal O_i]) = i;
		\]
		
		\item edges are of three types:
		\begin{itemize}
			\item horizontal edges preserve $\End(E)$;
			\item descending edges multiply the conductor by $\ell$;
			\item ascending edges divide the conductor by $\ell$;
		\end{itemize}
		and each vertex with $i>0$ has a unique ascending neighbour;
		
		\item the number of outgoing $\ell$-isogenies depends only on whether $\ell$
		divides the conductor:
		\begin{itemize}
			\item if $\ell\nmid [\mathcal O_K:\End(E)]$, then there are
			$1+\big(\frac{D}{\ell}\big)$ horizontal edges and
			$\ell-\big(\frac{D}{\ell}\big)$ descending edges;
			
			\item if $\ell\mid [\mathcal O_K:\End(E)]$, then there is exactly one ascending edge,
			no horizontal edges, and $\ell$ descending edges.
		\end{itemize}
	\end{enumerate}
\end{theorem}

\begin{definition}[$\ell$-volcano]
	An $\ell$-volcano is a connected graph whose vertices are partitioned into levels
	$V_0,\dots,V_d$ such that:
	\begin{itemize}
		\item each vertex in $V_i$ ($i>0$) has exactly one neighbour in $V_{i-1}$;
		\item edges either stay within a level or connect consecutive levels;
		\item vertices in $V_i$ for $i<d$ have degree $\ell+1$.
	\end{itemize}
	The level $V_0$ is called the surface and $V_d$ the floor.
\end{definition}

\subsection{Notation}

In the sequel, we work inside a fixed ordinary isogeny class over $\F_q$.
For a given $\ell$, we consider the decomposition
\[
V = \bigsqcup_{i=0}^d V_i,
\]
where all curves in $V_i$ have endomorphism ring $\mathcal O_i$.

In Section~\ref{sec:isogenyclass-distribution}, we introduce global (weighted)
distributions of endomorphism rings across the entire isogeny class.
\section{Order distribution inside one ordinary isogeny class over $\F_p$}
\label{sec:isogenyclass-distribution}

In this section we fix a \emph{single ordinary} isogeny class over a prime field
$\F_p$ and describe how endomorphism rings are distributed across the entire
isogeny class. In contrast with the local (componentwise) volcano picture,
this global viewpoint yields genuinely non-trivial distribution laws.

\subsection{The ordinary isogeny class and the conductor parameter}

Let $p$ be a prime and let $t\in\Z$ satisfy
\[
t^2<4p
\qquad\text{and}\qquad
p\nmid t.
\]
Then the isogeny class $\mathcal{I}(t,p)$ of elliptic curves $E/\F_p$ with
$\tr(\pi_E)=t$ is ordinary \cite{Waterhouse1969}.

Set
\[
\Delta:=t^2-4p<0,
\qquad
K:=\Q(\sqrt{\Delta}).
\]
Write uniquely
\begin{equation}\label{eq:Delta-factor}
	\Delta = v^2 D_K,
\end{equation}
where $D_K<0$ is the fundamental discriminant of $K$ and $v\ge 1$.

Let $\mathcal{O}_K$ denote the maximal order of $K$. For each divisor $f\mid v$,
define the quadratic order
\[
\mathcal{O}_f := \Z + f\,\mathcal{O}_K,
\qquad
D_f := \disc(\mathcal{O}_f)=f^2 D_K.
\]

For every $E\in\mathcal{I}(t,p)$ one has
\[
\Z[\pi_E]\subseteq \End(E)\subseteq \mathcal{O}_K,
\]
and $\End(E)\simeq \mathcal{O}_f$ for a unique divisor $f\mid v$
(Deuring \cite{Deuring1941}, Waterhouse \cite{Waterhouse1969}).

\subsection{Weighted class numbers and Deuring decomposition}

Let
\[
w(D):=\#\mathcal{O}_D^\times
\]
denote the number of units in the quadratic order of discriminant $D$, and define
the \emph{weighted class number}
\[
h^\ast(D):=\frac{h(D)}{w(D)}.
\]

The following theorem expresses the decomposition of the isogeny class in terms
of weighted class numbers.

\begin{theorem}[Weighted Deuring decomposition]\label{thm:deuring-decomposition}
	Let $p$ be prime and $t$ satisfy $t^2<4p$ and $p\nmid t$, and write
	$\Delta=t^2-4p=v^2D_K$. Then:
	\begin{enumerate}
		\item For every $E\in\mathcal{I}(t,p)$ one has $\End(E)\simeq \mathcal{O}_f$ for a unique
		$f\mid v$.
		
		\item The isogeny class admits the decomposition
		\begin{equation}\label{eq:hurwitz}
			H(\Delta)=\sum_{f\mid v} h^\ast(D_f),
		\end{equation}
		where $H(\Delta)$ is the Hurwitz class number.
		
		\item More precisely, the contribution of curves with endomorphism ring
		$\mathcal{O}_f$ is given by the weighted count $h^\ast(D_f)$.
	\end{enumerate}
\end{theorem}

\begin{proof}
	This follows from Deuring’s correspondence between ordinary isogeny classes
	over $\F_p$ and equivalence classes of positive definite binary quadratic forms
	of discriminant $\Delta$ (see \cite{Deuring1941,Waterhouse1969,Cox2013}).
	The decomposition according to endomorphism rings corresponds to partitioning
	these classes according to their associated orders.
\end{proof}

\begin{remark}
	The quantities $h^\ast(D_f)$ should be interpreted as \emph{weighted counts},
	reflecting the contribution of automorphism groups. In particular, they do not
	always correspond to literal cardinalities of sets of isomorphism classes.
\end{remark}

\subsection{Global weighted densities}

The previous decomposition naturally defines global distributions inside the
isogeny class.

\begin{definition}[Weighted exact density]
	For $f\mid v$, define
	\[
	\Delta^{=,\mathrm{wt}}_f(t,p)
	:=
	\frac{h^\ast(D_f)}{\sum_{g\mid v} h^\ast(D_g)}.
	\]
\end{definition}

\begin{definition}[Weighted containment density]
	For $f\mid v$, define
	\[
	\Delta^{\supseteq,\mathrm{wt}}_f(t,p)
	:=
	\sum_{\substack{g\mid v\\ f\mid g}}
	\Delta^{=,\mathrm{wt}}_g(t,p).
	\]
\end{definition}

\begin{proposition}[Explicit global distributions]
	For each $f\mid v$,
	\[
	\Delta^{=,\mathrm{wt}}_f(t,p)
	=
	\frac{h^\ast(D_f)}{\sum_{g\mid v} h^\ast(D_g)},
	\qquad
	\Delta^{\supseteq,\mathrm{wt}}_f(t,p)
	=
	\frac{\sum_{g\mid v,\ f\mid g} h^\ast(D_g)}{\sum_{g\mid v} h^\ast(D_g)}.
	\]
\end{proposition}

\begin{proof}
	This follows directly from the decomposition \eqref{eq:hurwitz} and the inclusion
	relation $\mathcal{O}_f \subseteq \mathcal{O}_g \iff f\mid g$.
\end{proof}

\subsection{Vertical $\ell$-adic distribution}

Fix a prime $\ell\neq p$. The $\ell$-adic valuation of the conductor
controls the position of a curve inside its $\ell$-volcano component.

Define the \emph{level mass}
\[
M_i(t,p;\ell)
:=
\sum_{\substack{f\mid v\\ v_\ell(f)=i}} h^\ast(D_f).
\]

\begin{corollary}[Global $\ell$-adic distribution]
	For every $i\ge 0$,
	\[
	\Pr\big[v_\ell(\mathrm{cond}(\End(E)))=i\big]
	=
	\frac{M_i(t,p;\ell)}{\sum_{g\mid v} h^\ast(D_g)}.
	\]
\end{corollary}

\begin{remark}
	This distribution provides a canonical global description of the vertical
	structure of $\ell$-isogeny volcanoes when averaged over all components.
\end{remark}
\section{Varying the prime: ring class fields and Chebotarev densities}
\label{sec:chebotarev}

Fix an imaginary quadratic discriminant $D<0$ and let
\[
\mathcal O_D = \Z + f\,\mathcal O_K
\]
be the quadratic order of discriminant $D=f^2D_K$ in the imaginary quadratic field
$K=\Q(\sqrt{D_K})$.

Let $H_D(X)\in\Z[X]$ denote the Hilbert class polynomial of discriminant $D$, and let
$L_D$ be the associated \emph{ring class field} of $\mathcal O_D$, i.e.\ the maximal
abelian extension of $K$ corresponding to the ideal class group $\Cl(\mathcal O_D)$.

One has
\begin{equation}\label{eq:ring-class-deg}
	[L_D:K]=h(D),\qquad [L_D:\Q]=2h(D),
\end{equation}
and there is a canonical isomorphism
\[
\Gal(L_D/K)\simeq \Cl(\mathcal O_D)
\]
via the Artin reciprocity map (see \cite[Ch.~9--11]{Cox2013}).

\subsection{CM existence over $\F_p$ and splitting of $H_D \bmod p$}

For a finite field $\F_q$ of characteristic $p$, define
\[
\Ell_{\mathcal O_D}(\F_q)
:=
\{\, j(E)\in \F_q : E/\F_q \text{ ordinary with } \End(E)\simeq \mathcal O_D \,\}.
\]

We recall the classical equivalence relating CM existence, Hilbert class polynomials,
and splitting in the ring class field.

\begin{theorem}[CM existence and splitting criterion]\label{thm:CM-splitting}
	Let $p$ be a prime such that $p\nmid D$. The following are equivalent:
	\begin{enumerate}
		\item $\Ell_{\mathcal O_D}(\F_p)\neq\varnothing$;
		
		\item the polynomial $H_D(X)$ has a root in $\F_p$;
		
		\item $H_D(X)$ splits completely into linear factors in $\F_p[X]$;
		
		\item $p$ splits completely in the ring class field $L_D/\Q$.
	\end{enumerate}
	
	Moreover, when these conditions hold, the polynomial $H_D$ has exactly $h(D)$ distinct
	roots in $\F_p$, corresponding to $h(D)$ distinct $j$-invariants.
\end{theorem}

\begin{proof}
	Assume $p\nmid D$. By Deuring's reduction theory, an elliptic curve over $\F_p$
	has endomorphism ring $\mathcal O_D$ if and only if the Frobenius at $p$ acts trivially
	on the set of CM $j$-invariants of discriminant $D$.
	
	By class field theory, this is equivalent to $p$ splitting completely in the ring
	class field $L_D$. The Galois action on the roots of $H_D$ is identified with
	$\Gal(L_D/K)$, and the factorisation of $H_D \bmod p$ is determined by the Frobenius
	conjugacy class.
	
	In particular, $H_D$ has a root in $\F_p$ if and only if Frobenius acts trivially,
	which is equivalent to complete splitting. In this case, all $h(D)$ roots lie in $\F_p$.
	See \cite{Deuring1941,Waterhouse1969,Cox2013,Silverman2009}.
\end{proof}

\begin{remark}[Isomorphism classes]
	The set $\Ell_{\mathcal O_D}(\F_p)$ parametrises $j$-invariants rather than
	isomorphism classes of curves. Different curves with the same $j$-invariant may occur
	as quadratic twists, but this does not affect the counting of CM points.
\end{remark}

\begin{remark}[Prime powers]
	Let $q=p^n$ with $p\nmid D$. Then $\Ell_{\mathcal O_D}(\F_q)\neq\varnothing$
	if and only if the Frobenius element at $p$ in $\Gal(L_D/\Q)$ has order dividing $n$.
	Equivalently, $H_D(X)$ has a root in $\F_{p^n}$.
\end{remark}

\subsection{Chebotarev density}

We now describe the asymptotic frequency of primes $p$ for which CM by $\mathcal O_D$
occurs over $\F_p$.

\begin{theorem}[Chebotarev density for CM primes]\label{thm:chebotarev-CM}
	Let $D<0$ and let $L_D/\Q$ be the ring class field of $\mathcal O_D$.
	Then the set
	\[
	\mathcal P_D := \{\, p \text{ prime} : p\nmid D \text{ and } \Ell_{\mathcal O_D}(\F_p)\neq\varnothing \,\}
	\]
	has natural density
	\begin{equation}\label{eq:chebotarev-density}
		\dens(\mathcal P_D) = \frac{1}{[L_D:\Q]} = \frac{1}{2h(D)}.
	\end{equation}
\end{theorem}

\begin{proof}
	By Theorem~\ref{thm:CM-splitting}, the set $\mathcal P_D$ consists exactly of primes
	that split completely in $L_D/\Q$ (excluding finitely many primes dividing $D$).
	
	By the Chebotarev density theorem, primes whose Frobenius element lies in the identity
	conjugacy class have density $1/[L_D:\Q]$. Using \eqref{eq:ring-class-deg} gives
	\eqref{eq:chebotarev-density}. See \cite{Serre1981}.
\end{proof}

\begin{remark}[Effective form]
	Effective versions of Chebotarev yield, for
	\[
	\pi_D(x):=\#\{\,p\le x:\ p\in\mathcal P_D\,\},
	\]
	an asymptotic of the form
	\[
	\pi_D(x)
	=
	\frac{\Li(x)}{2h(D)}
	+
	O\!\left(x\,e^{-c\sqrt{\log x}}\right),
	\]
	for some explicit constant $c>0$ depending on $L_D$.
\end{remark}
\section{Conclusion}\label{sec:conclusion}

In this work, we have provided a unified arithmetic framework for the study of
endomorphism rings within ordinary elliptic isogeny classes over finite fields.
Building on the Deuring--Waterhouse classification, we showed that the possible
endomorphism rings $\End(E)$ are precisely the quadratic orders lying between
$\Z[\pi_E]$ and the maximal order $\mathcal O_K$, and are naturally indexed by
the divisors of the conductor.

Our main contribution is to move beyond the local, componentwise perspective of
$\ell$-isogeny volcanoes and to introduce a genuinely global viewpoint. By
expressing the distribution of endomorphism rings in terms of weighted class
numbers, we obtained explicit formulas for global densities across the entire
isogeny class. These distributions provide a canonical probabilistic description
of conductor values and, consequently, of the vertical structure of $\ell$-volcanoes
when averaged over all components.

On the horizontal axis, varying the prime $p$, we related the occurrence of a
fixed CM order $\mathcal O_D$ to splitting conditions in the associated ring
class field $L_D$. This yields a precise asymptotic frequency governed by the
Chebotarev density theorem, namely a natural density of $1/(2h(D))$ for primes
admitting CM by $\mathcal O_D$. This establishes a clear bridge between the
arithmetic of class fields and the statistical behaviour of isogeny classes.

Taken together, these results highlight two complementary distribution laws:
a vertical, conductor-driven structure inside a fixed isogeny class, and a
horizontal, Chebotarev-governed distribution as the base field varies. This
duality provides a coherent framework for understanding the arithmetic and
combinatorial properties of isogeny graphs.

Beyond their intrinsic theoretical interest, these results have potential
applications in computational number theory and in isogeny-based cryptography,
where understanding the distribution of endomorphism rings is crucial for
algorithmic navigation, parameter selection, and security analysis.

Several directions for future work naturally arise. One may seek to refine these
distributions at the level of individual volcano components, to obtain
non-weighted counting results under additional arithmetic constraints, or to
extend the analysis to higher-dimensional abelian varieties and more general
isogeny graphs.

\medskip
\section*{Funding}
No external funding was received for this work.

\section*{Data Availability}
No new datasets were generated or analysed in this study.
All data needed to evaluate the results are contained in the article.

\section*{Code Availability}
Any computational checks can be reproduced using standard computer algebra systems
(e.g.\ SageMath/PARI) following the algorithms described in the text.  No new software
package was created for this work.

\section*{Conflicts of Interest}
The author declares no conflict of interest.

\section*{Ethical Approval}
Not applicable.

\section*{Consent to Participate}
Not applicable.

\section*{Consent for Publication}
Not applicable.

\section*{Author Contributions}
Conceptualization, methodology, formal analysis, writing (original draft), and writing
(review and editing): M.~El~Baraka and S. Ezzouak.


\end{document}